\documentclass[12pt, a4paper]{article}

\usepackage{euscript}
\usepackage{a4wide}
\usepackage{amsfonts}
\usepackage{amssymb, amsthm}
\usepackage[sumlimits, intlimits]{amsmath}
\usepackage{mathtools}
\usepackage{needspace}
\usepackage{etoolbox}
\usepackage{lipsum}
\usepackage{comment}
\usepackage{cmap}
\usepackage[pdftex]{graphicx}
\usepackage{hyperref}

\newcommand{\supp}{\mathop{\rm supp}\nolimits}
\newcommand{\lip}{\mathop{\rm Lip}\nolimits}

\theoremstyle{plain}
\newtheorem{thm}{Theorem}
\newtheorem{lm}{Lemma}

\newtheorem*{prop}{Proposition}

\theoremstyle{definition}
\newtheorem*{defin}{Definition}

\theoremstyle{remark}
\newtheorem*{rem}{Remark}

\newcommand{\footremember}[2]{%
    \footnote{#2}
    \newcounter{#1}
    \setcounter{#1}{\value{footnote}}%
}
\newcommand{\footrecall}[1]{%
    \footnotemark[\value{#1}]%
} 

\title{On displacement of viscous liquid in a system of parallel tubes}
\author{
Grigorii V. Monakov 
\footremember{SPbSU}{Saint-Petersburg State University, 7/9 Universitetskaya nab., St. Petersburg,199034, Russia}
\footremember{1}{email: st049008@student.spbu.ru}
\and 
Sergey B. Tikhomirov
\footrecall{SPbSU} 
\footremember{3}{email: s.tikhomirov@spbu.ru}
\footremember{4}{Corresponding Author}
\and 
Andrey A. Yakovlev
\footremember{NTC}{Gazpromneft Science \& Technology Centre, 75-79 liter D Moika River emb., St Petersburg 190000}
\footremember{5}{email: yakovlev.aale@gazpromneft-ntc.ru}
}

\begin{document}

\maketitle

\begin{abstract}
An explicitly solvable quasi 1D model of oil displacement is studied. The problem of recovering of the reservoir geometry is solved by means of a fixed point algorithm. The stability of solution is studied in various functional classes.
\end{abstract}

\section{Introduction}

The process of secondary oil recovery by pumping a fluid which displaces oil is known to be difficult to describe rigorously \cite{Barenbladt, bear}. In this note we consider a simple explicitly solvable quasi-1D model. Within this model the system is represented by a set of noninteracting tubes which can differ by their length and cross-section. The latter may be thought of as a way to take into account the number of tubes of the same length.
The left end of all tubes is identified with the pumping well and the right end -- with the production well. Each tube is separated into two segments filled with the corresponding immiscible phase (water or oil). The pressure difference between the wells is the parameter in the problem. The tubes are assumed to be homogeneous so that the pressure in each tube is linear with respect to the length. 

The problem studied in the paper is an inverse one -- given the amount of oil extracted at the producing well as a function of the amount of water pumped in, find the geometry of the the reservoir, that is, determine the lengths and cross-sections of the tubes. Questions of this kind are known as history matching in the context of petroleum engineering \cite{guptaking}. 

The model we consider comes from the Dykstra--Parsons model \cite{Dykstra}. A partial case of the model we study is studied in \cite[Part 4]{guptaking} corresponding to the pressure difference constant in time.  Notice that the recovery problem appears to have not been analyzed previously. Note that representation of the porous media as a complex system of tubes are quite often used in the literature, see for instance derivation of capillarity pressure in Buckley-Leverett model \cite{BL} and relative permeability calculation \cite{Stone1970, Stone1973}.

The solution of the problem above naturally splits into the questions of uniqueness, stability (well-posedness) and that of finding an actual recovery procedure. There is an obvious scaling non-uniqueness in  problem (see section 4), hence it remains to study whether the solution is unique after fixing the scaling parameter. Our main results are theorems 2 and 4. Theorem 2 establishes the uniqueness of solution upon  choosing the scaling. The proof of it is based on a fixed point theorem. The geometry of the reservoir is then recovered via finding the fixed point of an explicitly given operator, hence theorem 2 also provides a recovery procedure. Theorem 4 establishes the stability with respect to small perturbations in appropriate functional classes for the above problem. This assertion is rather important since the precision of debits' measurements on pumping and producing wells is notoriously low, so only a stable recovery procedure is valuable.

The structure of the paper is as follows. In section 2 we formulate the model for one tube and solve the basic equation for the phase boundary. In section 3 we calculate the water and oil debits as functions of time for a finite system of tubes, respectively. In section 4 we describe the continuum limit $ n \to \infty $ ($n $ being the number of tubes) of the model and explain the scaling non-uniqueness. Section 5 contains the formulation and proof of the main uniqueness and recovery result. Section 6 is devoted to study of the stability of the problem.

\section{Description of model with one tube}

Consider two wells connected with one thin tube of length $L$ and cross-section $S$. At the moment $t = 0$ this tube is filled with oil of viscosity $\mu_o$. Then we start to pump water of viscosity $\mu_w$ into the pumping well, located at $ x=0 $. The pressure difference between two wells is a given function of time, to be denoted $\Delta p(t)$. Since we think that all liquids in our problem are incompressible, at the very first moment oil and after a while water will start to flow out from second well. Let's denote oil flow rate with $Q_o(t)$ and water flow rate with $Q_w(t)$. Under assumption that our tube is thin enough ($L^2 \gg S$) we think that all characteristics only depend on $x$ coordinate, which means that we have only one point in our tube, where oil contacts water (lets denote distance between the first well and this point by $l(t)$) and pressure $p(t, x)$ and speed $v(t, x)$ of liquid in the tube depend only on one coordinate $x$ and time $t$. Since our flow satisfies continuity equation, it's easy to see that $v(t, x_1) = v(t, x_2)$ for all $x_1, x_2 \in [0, L]$, which means that flow speed $v(t)$ depends only on time $t$, and $\dfrac{dl}{dt}(t) = v(t)$. The Darcy's law for this model has the form,

\begin{gather*}
v(t) = -\frac{k}{\mu_w} \cdot \frac{\partial p}{\partial x}(t, x)
\end{gather*}

Here $k$ is a positive constant (permeability). This equation holds for all $x \le l(t)$. Consider now $x > l(t)$. Since we have a different liquid here (oil instead of water) the equation will have another viscosity in the right part:

\begin{gather*}
v(t) = -\frac{k}{\mu_o} \cdot \frac{\partial p}{\partial x}(t, x)
\end{gather*}

Integrating the first equation from $x = 0$ to $x = l(t)$ and second equation from $x = l(t)$ to $x = L$ and summing up the results with coefficients $\mu_w$ and $\mu_o$ respectively we get:

\begin{gather*}
v(t) \cdot (l(t) \cdot \mu_w + (L - l(t)) \cdot \mu_o) = -k (p(t, l(t)) - p(t, 0)) - k (p(t, L) - p(t, l(t)))
\end{gather*}

Since $p(t, 0) - p(t, L) = \Delta p(t)$ this can be written in the form,

\begin{gather*}
\frac{dl}{dt}(t) = \frac{k \cdot \Delta p (t)}{l(t) \cdot \mu_w + (L - l(t)) \cdot \mu_o}
\end{gather*}

For brevity let us denote $\kappa = \dfrac{\mu_w}{\mu_o} < 1$ and $c(t) = \dfrac {k}{\mu_o} \cdot \Delta p(t)$. Now we have an ordinary differential equation which can be solved:

\begin{gather}
\label{tubes1_2}
\frac{dl}{dt}(t) = \frac{c(t)}{l(t) \cdot \kappa + (L - l(t))}
\end{gather}

\begin{gather*}
\frac{\kappa - 1}{2} \cdot l(t)^2 + L \cdot l(t) + c_1 = \int\limits_0^t c(\tau) \ d\tau 
\end{gather*}

Since $l(0) = 0$ we get $c_1 = 0$. Another useful notation will be $F(t) = \int_0^t c(\tau) \ d\tau$. Now let's solve this quadratic equation. Since we know that $l(t) \le L$ there is only one possible root:

\begin{equation}
\label{tubes1}
l(t) = \frac{L - \sqrt{L^2 - 2 (1 - \kappa) \cdot F(t)}}{1 - \kappa}
\end{equation}

We will also need some information about $F(t^*)$, where $t^* = \sup \{t \in [0, +\infty) : l(t) < L\}$:

\begin{gather}
\notag 
L = \frac{L - \sqrt{L^2 - 2(1-\kappa)F(t^*)}}{1 - \kappa}\\
\notag 
\kappa^2 L^2 = L^2 - 2(1-\kappa)F(t^*)\\
\label{tubes1_1}
F(t^*) = \frac{1+\kappa}{2} \cdot L^2
\end{gather}

\section{Description and properties of model with many tubes}

Let us now consider the system of $n$ parallel thin tubes with lengths $L_1, L_2, \dots, L_n$ and respective cross-sections $S_1, S_2, \dots, S_n$. At the moment $t = 0$ all tubes are filled with oil, and then, as in our previous model, we start to pump water into the first well. Let us denote the volume of oil and water extracted in first $t$ seconds by $\tilde V_o(t)$ and $\tilde V_w(t)$, respectively. Then we have an obvious connection between $\tilde V_o(t)$, $\tilde V_w(t)$ and $Q_o(t)$, $Q_w(t)$:

\begin{gather*}
\tilde V_o(t) = \int\limits_0^t Q_o(\tau) d\tau, \quad \quad \tilde V_w(t) = \int\limits_0^t Q_w(\tau) d\tau
\end{gather*}

It is clear that we can also calculate the amount of water at the moment $t$ using the following formula:

\begin{gather*}
\tilde V_o(t) = l_1(t) \cdot S_1 + l_2(t) \cdot S_2 + ... + l_n(t) \cdot S_n
\end{gather*}

Suppose that $L_1 < L_2 < \dots < L_n$. Let's denote with $t_k$ the moment when $k$-th tube starts leaking water, or more precisely $t_k = \sup \{t \in [0, +\infty):l_k(t) \le L_k\}$. According to formula $\eqref{tubes1}$ we can see that $l_k(t)$ depends only on one parameter of the tube --- length $L_k$, and then formula $\eqref{tubes1_1}$, inequality $L_1 < L_2 < \dots < L_n$ and monotonicity of the function $F$ imply $t_1 < t_2 < \dots < t_n$. Then on the segment $[t_k, t_{k + 1})$ the following formula holds,
\begin{gather*}
\tilde V_o(t)  - \tilde V_o(t_k) = \sum_{j=k+1}^n (l_j(t) - l_j(t_k)) \cdot S_j ,
\end{gather*}
since tubes with numbers in $\{1, 2, \dots, k\}$ are already filled with water and do not contribute in the oil production anymore. Let us substitute formula $\eqref{tubes1}$ into the last equation:

\begin{gather*}
\tilde V_o(t) - \tilde V_o(t_k) = \sum_{j=k+1}^n \left(\frac {L_j - \sqrt{L_j^2 - 2(1-\kappa) \cdot F(t)}} {1 - \kappa} - \frac {L_j - \sqrt{L_j^2 - 2(1-\kappa) \cdot F(t_k)}} {1 - \kappa}\right) \cdot S_j = \\ = \sum_{j=k+1}^n \frac {\sqrt{L_j^2 - 2(1-\kappa) \cdot F(t_k)} - \sqrt{L_j^2 - 2(1-\kappa) \cdot F(t)}} {1 - \kappa} \cdot S_j
\end{gather*}

We can get a similar formula for the water flow rate (note, that when $t_k \le t \le t_{k+1}$ holds, water goes only through tubes with numbers $\{1, 2, \dots, k\}$):

\begin{gather*}
Q_w(t) = \sum_{j=1}^k \frac{dl_j}{dt}(t) \cdot S_j
\end{gather*}

And according to $\eqref{tubes1_2}$, since $l_j(t) = L_j$ for $t_k \le t \le t_{k+1}$ and $1 \le j \le k$ we get:

\begin{gather*}
Q_w(t) = \sum_{j=1}^k \frac{c(t)} {L_j \cdot \kappa} \cdot S_j = \frac{c(t)} {\kappa} \sum_{j=1}^k \frac{S_j} {L_j}\\
\tilde V_w(t) - \tilde V_w(t_k) = \frac{F(t) - F(t_k)} {\kappa} \sum_{j=1}^k \frac{S_j} {L_j},\ \text{where $t_k \le t \le t_{k+1}$}
\end{gather*}

Returning to formula $\eqref{tubes1_1}$ and substituting $t^* = t_k$ we get:

\begin{gather*}
F(t_k) = \frac{1+\kappa}{2} \cdot L_k^2
\end{gather*}

To sum up, in this section we got two formulas (for $t_k \le t \le t_{k+1}$) we will use below:

\begin{gather}
\label{tubes2}
\tilde V_o(t)  - \tilde V_o(t_k) = \sum_{j=k+1}^n \frac {\sqrt{L_j^2 - 2(1-\kappa) \cdot \frac{1+\kappa}{2} \cdot L_k^2} - \sqrt{L_j^2 - 2(1-\kappa) \cdot F(t)}} {1 - \kappa} \cdot S_j\\
\label{tubes3}
\tilde V_w(t) - \tilde V_w(t_k) = \frac{F(t) - \frac{1+\kappa}{2} \cdot L_k^2} {\kappa} \sum_{j=1}^k \frac{S_j} {L_j}
\end{gather}

Note, that we can exclude $F(t)$ from this equations and get a parametric equality, which allows us to think of a curve $(\tilde V_w(t), \tilde V_o(t))$, which only depends on the sets $L_1, L_2, \dots, L_n$ and $S_1, S_2, \dots, S_n$, but not on $F(t)$. This means that our curve does not change if we take any monotonic smooth reparameterization $\alpha : [0, +\infty) \to [0, +\infty)$.

\begin{rem}
Despite of simplicity the model mimics such a complicated phenomenon as viscous fingerings \cite{ST}. Assume that we have a family of tubes of similar, but slightly different length. The speed of water propagation will be higher in short tubes due to less amount of liquid in them. With further propagation of water, the average viscosity in short tubes becomes smaller comparing to average viscosity in long tubes and hence the difference in speed between short and long tubes increases. Such a process corresponds to growth of viscous fingers. At the same time the model cannot cover other phenomena such as subdivision of fingers, changing topology of water pond etc. 
\end{rem}

\section{Continuous limit}

We would now like to pass to the limit of infinite system of tubes to model a continuous environment. To reach this let's consider a measure $\mu$ with bounded support in $(0, +\infty)$. We will think of physical meaning of this measure in a following way -- the measure of a line subset $A$ is equal to a sum of cross-sectional areas of all tubes with lengths in $A$. Then our previous model with $n$ tubes can be presented as measure $\mu = \sum_{k=1}^n S_k \delta_{L_k}$, where 
$\delta_L(A) = \begin{cases}
1,&\text{if $L \in A$;}\\
0,&\text{if $L \notin A$.}\\
\end{cases}$

We want to understand formulae (\ref{tubes2}) and (\ref{tubes3}) as refereing to a discrete measure $ \mu $ just described and extend them by continuity to all measures. More precisely, we want maps which takes a measure $\mu$ into two functions, $ V_w $ and $ V_o $, so that in the case of measure $ \mu $ of the form $\mu = \sum\limits_{k=1}^n S_k \cdot \delta_{L_k}$ the resulting functions $ V_w $ and $ V_o $ are a reparametrization of (\ref{tubes2}) and (\ref{tubes3}), that is, there exists a continuous bijection $ \xi $, such that $ \tilde V_o ( \xi (t )) = V_o ( t ) $, $ \tilde V_w ( \xi (t )) = V_w( t ) $. We would like the maps to be continuous from the space of measures endowed with the weak-$ * $ topology into the space of continuous functions with the standard metric.  

\begin{lm}
The maps 
\[ \mu \mapsto V_w , \mu \mapsto V_o \]
defined by the formulae
\begin{gather}
\label{def_w}
V_w(\alpha) = \frac{1 + \kappa}{\kappa}\int\limits_0^{\alpha} t \cdot \int\limits_0^t \frac {1}{y}\ d\mu(y)\ dt , \\
\label{def_o}
V_o(\alpha) = (1 + \kappa) \int\limits_{0}^{\alpha} t \cdot \int\limits_{t}^{\infty} \frac{1} {\sqrt{y^2 - (1 - \kappa^2)t^2}}\ d\mu(y)\ dt
\end{gather}
satisfy 
\[ 
V_w \left( \sqrt { \frac 2{1+k} F ( t ) } \right) = \tilde V_w ( t ) , 
V_o \left( \sqrt { \frac 2{1+k} F ( t ) } \right) = \tilde V_o ( t ) , 
\] 
for $\mu = \sum\limits_{k=1}^n S_k \cdot \delta_{L_k}$, are continuous from the space of measures endowed with the weak-$*$ topology into $ C ( 0 , M ) $ for any $ M > 0 $.  
\end{lm}

\begin{proof}
The maps (\ref{def_w}) and \eqref{def_o} are obviously continuous in the described topology. 
Let $\mu = \sum\limits_{k=1}^n S_k \cdot \delta_{L_k}$. Then
for $ \alpha \in [ L_k , L_{ k+1 }] $ we have
\[ V_w ( \alpha ) = \frac{ 1+\kappa}{\kappa} \left( \sum_{ j=1 }^k \frac { S_j }{ L_j } \frac{ L_{ j }^2 - L_{j-1}^2 }2 + \sum_{ j=1 }^k \frac { S_j }{ L_j } \frac{ \alpha^2 - L_{j-1}^2 }2 \right) . \]
This implies (\ref{tubes2}), for $ \sqrt { \frac 2{1+k} F ( t ) } = L_j $. Formula (\ref{tubes3}) is verified in a similar way.  

\end{proof}

The choice of the topology on measures we have made appears natural because the  set of $\delta$-measures is total, hence with this choice the extension of the functions $ \tilde V_w $ and $ \tilde V_o $ described in this lemma is unique.

\begin{defin} The function $ V_w $ defined by (\ref{def_w})is non-decreasing, the function $ V_o $ defined by (\ref{def_o}) is increasing, hence $ \{ (V_w(\alpha), V_o(\alpha) + V_w(\alpha)) \}_{ \alpha > 0 } $ is a graph of a monotonic function. We will denote this monotonic function by $ \mathcal L ( \mu )$. 
\end{defin}
The function $\mathcal L ( \mu )$ has an important role in applications it is called displacement characteristic and shows how fraction of water in the extracted liquid changes in time.

In what follows we study the question of hystory matching i.e. can we find the ``environment'' (measure $\mu$) by the known displacement characteristic $\mathcal L ( \mu )$.

Notice also that formulae (\ref{def_w}) and (\ref{def_o}) can be simplified as follows,

\begin{gather*}
V_w(\alpha) = \frac{1 + \kappa}{\kappa}\int\limits_0^{\alpha} \int\limits_0^t \frac {t}{y}\ d\mu(y)\ dt = \\ = \frac{1 + \kappa}{\kappa} \int\limits_0^{\alpha} \int_y^{\alpha} \frac {t}{y}\ dt\ d\mu(y) =  \frac{1 + \kappa}{\kappa} \int\limits_0^{\alpha}\left. \frac{t^2}{2y}\right |_y^{\alpha} d\mu(y) = \\ = \frac{1 + \kappa}{\kappa} \int\limits_0^\alpha \left( \frac{\alpha^2}{2y} - \frac{y}{2} \right)\ d\mu(y) = \frac{1 + \kappa}{2\kappa} \int\limits_0^\alpha \frac{\alpha^2 - y^2}{y} d\mu(y)
\end{gather*}

\begin{gather*}
V_o(\alpha) = (1 + \kappa) \int\limits_{0}^{\alpha} \int\limits_{t}^{\infty} \frac{t} {\sqrt{y^2 - (1 - \kappa^2)t^2}}\ d\mu(y)\ dt = \\ = (1 + \kappa) \cdot \int\limits_{0}^{\infty} \int\limits_{0}^{\min(y, \alpha)} \frac{t} {\sqrt{y^2 - (1 - \kappa^2)t^2}}\ dt\ d\mu(y) = \\ = \frac{1}{1 - \kappa} \cdot \int\limits_{0}^{\infty} \int\limits_{0}^{\min(y, \alpha)} \frac{(1 - \kappa^2) \cdot t} {\sqrt{y^2 - (1 - \kappa^2)t^2}}\ dt\ d\mu(y) = \\ = \frac{-1}{1 - \kappa} \int_0^{\infty} \left. \sqrt{y^2 - (1 - \kappa^2)) \cdot t^2} \right|_{0}^{\min(y, \alpha)}\ d\mu(y) = \\ = \int\limits_0^\alpha y\ d\mu(y) + \frac{1}{1 - \kappa} \int\limits_\alpha^\infty (y - \sqrt{y^2 - (1 - \kappa^2)\cdot \alpha^2})\ d\mu(y)
\end{gather*}

For future reference, we write down explicitly the resulting formulae for the functions $V_o(\alpha), V_w(\alpha)$,

\begin{gather}
\label{tubes4}
V_w(\alpha) = \frac{1 + \kappa}{2\kappa}\int\limits_0^{\alpha} \frac{\alpha^2 - y^2}{y} d\mu(y) \\
\label{tubes5}
V_o(\alpha) = \int\limits_0^\alpha y d\mu(y) + \frac{1}{1 - \kappa} \int\limits_\alpha^\infty \left(y - \sqrt{y^2 - (1-\kappa^2) \alpha^2}\right) d\mu(y)
\end{gather}

The problem of recovering the measure $ \mu $ from the curve $ \mathcal L ( \mu ) $ has a scaling non-uniqueness described in the following

\begin{rem} Suppose that two measures $\mu_1, \mu_2$ satisfy $\mu_1(A) = k \cdot \mu_2(k \cdot A)$ for all $A \subset \mathbb{R}_+$, where $k \in \mathbb{R}_+$ and $k \cdot A = \{x \in \mathbb{R}: \frac{x}{k} \in A\}$. Then the curves $\mathcal L (\mu_1)=\mathcal L (\mu_2)$.
\end{rem}

\begin{proof}
We have
\begin{gather*}
V_{w, \mu_1}(\alpha) = \frac{1 + \kappa}{2\kappa} \int\limits_0^\alpha \frac{\alpha^2 - y^2}{y}\ d\mu_1(y) = \frac{1 + \kappa}{2\kappa} \int\limits_0^\alpha \frac{\alpha^2 - y^2}{y}\ k\ d\mu_2(ky) = \\ = \frac{1 + \kappa}{2\kappa} \int\limits_0^\alpha \frac{(k\alpha)^2 - (ky)^2}{ky}\ d\mu_2(ky) 
\end{gather*}
Denote $x = ky$:

\begin{gather*}
V_{w, \mu_1}(\alpha) = \frac{1 + \kappa}{2\kappa} \int\limits_0^{k\alpha} \frac{(k\alpha)^2 - (x)^2}{x}\ d\mu_2(x) = V_{w, \mu_2}(k\alpha)
\end{gather*}

Now let us do the same to $V_{o, \mu_1}$:

\begin{gather*}
V_{o,\mu_1}(\alpha) = \int\limits_0^\alpha y\ d\mu_1(y) + \frac{1}{1 - \kappa} \int\limits_\alpha^\infty \left(y - \sqrt{y^2 - (1 - \kappa^2)\cdot \alpha^2} \right)\ d\mu_1(y) = \\ = \int\limits_0^\alpha yk\ d\mu_2(ky) + \frac{1}{1 - \kappa} \int\limits_\alpha^\infty \left(y - \sqrt{y^2 - (1 - \kappa^2)\cdot \alpha^2}\right)k\ d\mu_2(ky) = \\ = \int\limits_0^\alpha (ky) \ d\mu_2(ky) + \frac{1}{1 - \kappa} \int\limits_\alpha^\infty \left(ky - \sqrt{(ky)^2 - (1 - \kappa^2)\cdot (k\alpha)^2}\right)\ d\mu_2(ky)
\end{gather*}

Denote $x = ky$:

\begin{gather*}
V_{o,\mu_1}(\alpha) = \int\limits_0^{k\alpha} x\ d\mu_2(x) + \frac{1}{1 - \kappa} \int\limits_{k\alpha}^\infty \left(x - \sqrt{x^2 - (1 - \kappa^2)\cdot (\alpha k)^2}\right)\ d\mu_2(x) = V_{o, \mu_2}(k\alpha)
\end{gather*}

We thus see that the curve $\mathcal L(\mu_1)$ is obtained by a reparametrization of $\mathcal L(\mu_2)$. \end{proof}

\section{Uniqueness of measure, describing the limit model with given curve}

As remarked in the previous section for any nonzero measure $ \mu $ the curve $(V_w(\alpha), V_o(\alpha) + V_w(\alpha))$ is a graph of a monotonic function, $G: [0, \infty) \to [0, \infty)$,

\begin{gather}
\label{tubes6}
V_w(\alpha) = G(V_o(\alpha) + V_w(\alpha)) , \alpha \ge 0 . 
\end{gather}

Notice that $ G ( s ) \to \infty $ as $ s\to \infty $ for $V_w(\alpha) \to \infty $ as $ \alpha \to \infty $. In fact $ G $ is a Lipschitz function.

\begin{lm}
Function $G$ described above satisfies $G \in Lip([0, +\infty))$ and $Lip(G) \le 1$, which is:
\begin{gather*}
|G(x) - G(y)| \le |x - y|
\end{gather*}
for all $x, y \in [0, +\infty)$.
\end{lm}

\begin{proof}
Since $V_w + V_o$ is a monotonic bijection from $[0, +\infty)$ onto itself for every $x, y \in [0, +\infty)$ we can take two points $\alpha_1 < \alpha_2$ such that $V_w(\alpha_1) + V_o(\alpha_1) = x$ and $V_w(\alpha_2) + V_o(\alpha_2) = y$. Let's write $\eqref{tubes6}$ for them and deduct the first one from the second one:

\begin{gather*}
G(V_o(\alpha_2) + V_w(\alpha_2)) - G(V_o(\alpha_1) + V_w(\alpha_1)) = V_w(\alpha_2) - V_w(\alpha_1) , 
\end{gather*}

And since $V_w$ and $V_o$ increase we get:

\begin{gather*}
V_w(\alpha_2) - V_w(\alpha_1) \le (V_w(\alpha_2) - V_w(\alpha_1)) + (V_o(\alpha_2) - V_o(\alpha_1)) = \\ = \left|V_w(\alpha_2) + V_o(\alpha_2) - V_w(\alpha_1) - V_o(\alpha_1)\right|
\end{gather*}

Hence:

\begin{gather*}
G(V_o(\alpha_2) + V_w(\alpha_2)) - G(V_o(\alpha_1) + V_w(\alpha_1)) \le \left|(V_w(\alpha_2) + V_o(\alpha_2)) - (V_w(\alpha_1) + V_o(\alpha_1))\right| , 
\end{gather*}
that is
\begin{gather*}
|G(x) - G(y)| \le |x - y|
\end{gather*}
\end{proof}

Now we will assume that we are given a function $G$ and want to solve $\eqref{tubes6}$ as an equation for $\mu$. Note, that in previous section we have showed that if this equation has a solution $\mu$, then all measures $\mu_k(A) = k^2 \cdot \mu(k \cdot A)$ are also solutions. We will later show how to fix one more numerical parameter in addition to function $G$ to get a unique solution. At first it will be useful to exclude $\mu$ and express $V_o(\alpha) + V_w(\alpha)$ in terms of $V_w(\alpha)$ to solve $\eqref{tubes6}$ as an equation for $V_w(\alpha)$ and then reestablish $\mu$ from $V_w(\alpha)$.

Since our measure $\mu$ has a bounded support let's consider a point $\alpha_{max} \in (0, +\infty)$ such that if $A \subset [0, +\infty)$ and $A \cap [0, \alpha_{max}) = \emptyset$ then $\mu(A) = 0$. Define
\[ R ( \alpha ) = \sqrt{ \alpha_{ max}^2 - ( 1 - \kappa^2 ) \alpha^2 } . \]
Then

\begin{lm}
If functions $V_w(\alpha), V_o(\alpha)$ satisfy $\eqref{tubes4}$, $\eqref{tubes5}$ then for $\alpha \in [0, \alpha_{max}]$ following equation holds,

\begin{gather} \label{VoVw}
V_o(\alpha) + V_w(\alpha) = \frac{\kappa}{1 - \kappa^2} \left(\alpha_{max} - R ( \alpha ) \right) \cdot V_w'(\alpha_{max}) + \\ + \frac{\kappa}{1 - \kappa^2} \cdot \left(\frac{\alpha_{max}^2 + \left( R ( \alpha ) \right)^2}{\alpha_{max} R (\alpha)} - 2\right) \cdot V_w(\alpha_{max}) + \\ + \kappa (1 - \kappa^2) \alpha^4 \cdot \int\limits_{\alpha}^{\alpha_{max}} \frac{V_w(y)}{y^2 \cdot (y^2 - (1 - \kappa^2) \cdot \alpha^2)^{\frac{3}{2}}}\ dy 
\end{gather} 
\end{lm}

Note that $V_w$ is smooth for $\alpha \ge \alpha_{max}$.

\begin{proof}
To prove this fact we just substitute $\eqref{tubes4}$ and $\eqref{tubes5}$ into this equation.

\end{proof}

Let us now substitute (\ref{VoVw}) to $\eqref{tubes6}$,

\begin{gather*}
V_w(\alpha) = G \left( \frac{\kappa}{1 - \kappa^2} \left(\alpha_{max} - R ( \alpha) \right) V_w'(\alpha_{max}) + \frac{\kappa}{1 - \kappa^2} \frac{\left( \alpha_{max} - R ( \alpha ) \right)^2}{\alpha_{max} R (\alpha)} V_w(\alpha_{max}) + \right. \\ \left. + \kappa (1 - \kappa^2) \alpha^4 \cdot \int\limits_{\alpha}^{\alpha_{max}} \frac{V_w(y)}{y^2 (y^2 - (1 - \kappa^2) \cdot \alpha^2)^{\frac{3}{2}}}\ dy \right) .
\end{gather*} 

As has been said above, we have to fix one more numerical parameter besides the function $G$ to get a unique solution of this equation. Let's assume that we are given $\alpha_{max}$ and graph of $G$ on the section $[0, V_w(\alpha_{max}) + V_o(\alpha_{max})]$. It's easy to see that:

\begin{gather*}
V_w(\alpha_{max}) = \frac{(1 + \kappa) \cdot \alpha_{max}^2}{2\kappa}\int\limits_0^{\infty} \frac{1}{y} d\mu(y) - \frac{1 + \kappa}{2\kappa}\int\limits_0^{\infty} y d\mu(y) \\
V_o(\alpha_{max}) = \int\limits_0^\infty y d\mu(y) \\
V_w'(\alpha_{max}) = \frac{(1 + \kappa) \cdot \alpha_{max}}{\kappa}\int\limits_0^{\infty} \frac{1}{y} d\mu(y)
\end{gather*}

Since $ V_w(\alpha_{max}) + V_o(\alpha_{max}) $ and 
\[ G(V_w(\alpha_{max}) + V_o(\alpha_{max})) = V_w(\alpha_{max}) , \]
can be read off from the graph of $G$, one can recover $ V_w(\alpha_{max}) $ and $ V_w'(\alpha_{max}) $ from the graph,
\begin{eqnarray}
V_w(\alpha_{max}) = G(V_w(\alpha_{max}) + V_o(\alpha_{max})) \label{grw} \\
V_w'(\alpha_{max}) = 2 \cdot V_w(\alpha_{max}) + \frac{1 + \kappa}{\kappa} \cdot V_o(\alpha_{max}) .\label{grwprime}
\end{eqnarray}
We denote 
\begin{gather*}
h(\alpha) = \frac{\kappa}{1 - \kappa^2} \cdot \left(\alpha_{max} - R ( \alpha )\right) V_w'(\alpha_{max}) + \frac{\kappa}{1 - \kappa^2} \frac{\left( \alpha_{max} - R ( \alpha ) \right)^2}{\alpha_{max} R (\alpha)} V_w(\alpha_{max})
\end{gather*}
Using (\ref{grw}) and (\ref{grwprime}) this formula can be rewritten as follows,
\begin{gather} \label{h} 
h(\alpha) = \frac{\kappa}{1 - \kappa^2} \left(\alpha_{max} - R ( \alpha) \right) \left(2 \cdot G(V_{max}) + \frac{1 + \kappa}{\kappa} (V_{max} - G(V_{max}))\right) + \\ + \frac{\kappa}{1 - \kappa^2}  \cdot \frac{\left( \alpha_{max} - R ( \alpha ) \right)^2}{\alpha_{max} R (\alpha)} G(V_{max})
\end{gather}
Here $V_{max}$ stands for $V_o(\alpha_{max}) + V_w(\alpha_{max})$. Thus choosing the function $G$ at the point $V_{max}$ and the numerical parameter $\alpha_{max}$ fixes our function $h$.

Let us define an operator $T: L^\infty([0, \alpha_{max}]) \to L^\infty([0, \alpha_{max}])$ by
\begin{gather}
\label{tubes7}
(TV)(\alpha) = \kappa \cdot (1 - \kappa^2) \cdot \alpha^4 \cdot \int\limits_{\alpha}^{\alpha_{max}} \frac{V(y)}{y^2 (y^2 - (1 - \kappa^2) \cdot \alpha^2)^{\frac{3}{2}}}\ dy
\end{gather}

Now we are ready to state the main result.

\begin{thm}
Given a function $G$ and a value $\alpha_{max}$ as described above, the equation
\begin{gather*}
V(\alpha) = G(h(\alpha) + (TV)(\alpha))
\end{gather*}
has a unique solution with respect to function $V \in L^{\infty}(0, \alpha_{max})$.
\end{thm}

\begin{proof}
Our main goal is to prove that the operator
\[ \psi \mapsto G \circ (h + T\psi ) \] is a contraction on $L^\infty([0, \alpha_{max}])$. This will imply the existense and uniqueness of the solution by a standard theorem of functional analysis \cite{simon_reed}. For any functions $V_1, V_2 \in L^\infty([0, \alpha_{max}])$,
\begin{gather*}
\left|G(h(\alpha) + (TV_1)(\alpha)) - G(h(\alpha) + (TV_2)(\alpha))\right| \le \\
\left|(h(\alpha) + (TV_1)(\alpha)) - (h(\alpha) + (TV_2)(\alpha))\right| = \\
= \left|T(V_1 - V_2) (\alpha)\right|
\end{gather*}

since $\lip(G) \le 1$. Then

\begin{gather*}
\left|T(V_1 - V_2) (\alpha)\right| \le \kappa \cdot (1 - \kappa^2) \cdot \alpha^4 \cdot \int\limits_{\alpha}^{\alpha_{max}} \frac{\left|V_1(y) - V_2(y)\right|}{y^2 \cdot (y^2 - (1 - \kappa^2) \cdot \alpha^2)^{\frac{3}{2}}}\ dy
\end{gather*}

Denoting $x = \dfrac{y}{\alpha}$ we get:

\begin{gather*}
\left|T(V_1 - V_2) (\alpha)\right| \le \kappa \cdot (1 - \kappa^2) \int\limits_{1}^{\frac{\alpha_{max}}{\alpha}} \frac{\|V_1 - V_2\|_{L^{\infty}}}{x^2 \cdot (x^2 - (1 - \kappa^2))^{\frac{3}{2}}}\ dx \le \\ 
\le \kappa \cdot (1 - \kappa^2) \int\limits_{1}^{\infty} \frac{1}{x^2 \cdot (x^2 - (1 - \kappa^2))^{\frac{3}{2}}}\ dx \cdot \|V_1 - V_2\|_{L^{\infty}} = \\ 
= \kappa \cdot (1 - \kappa^2) \cdot \frac {2 - (1 - \kappa^2) - 2 \cdot \sqrt{1 - (1 - \kappa^2)}}{(1 - \kappa^2)^2 \cdot \sqrt{1 - (1 - \kappa^2)}} \cdot \|V_1 - V_2\|_{L^{\infty}} = \\ 
= \frac{1 + \kappa^2 - 2\kappa}{1 - \kappa^2} \cdot \|V_1 - V_2\|_{L^{\infty}} = 
\frac{1 - \kappa}{1 + \kappa} \cdot \|V_1 - V_2\|_{L^{\infty}}
\end{gather*}

And since $\dfrac{1 - \kappa}{1 + \kappa} < 1$ we can see that operator $G \circ (h + T)$ is contracting, which proves the theorem.
\end{proof}

Now we want to prove that given function $G$ and value $\alpha_{max}$ we have at most one possible measure $\mu$. As stated in the theorem, with this data we can get only one function $V_w$, which satisfies our equation. Hence, following theorem finishes the proof of uniqueness of measure.

\begin{thm}
Consider function $V : [0, \alpha_{max}] \to \mathbb{R}$, which satisfies $\eqref{def_w}$ for some measures $\mu_1, \mu_2$, such that supports of  $\mu_1, \mu_2$ lie in $(0, \alpha_{max})$, then $\mu_1 = \mu_2$.
\end{thm}

\begin{proof}
At first consider measure $\mu$ which satisfies 

\begin{gather*}
V(\alpha) = \frac{1 + \kappa}{\kappa}\int\limits_0^{\alpha} t \cdot \int\limits_0^t \frac {1}{y}\ d\mu(y)\ dt
\end{gather*}

Let's denote $F(t) = \int\limits_0^t \frac{1}{y}\ d\mu(y)$. Then:

\begin{gather*}
V(\alpha) = \frac{1 + \kappa}{\kappa} \int\limits_0^\alpha t  \cdot F(t)\ dt
\end{gather*}

And hence:

\begin{gather}
\label{tubes8}
V'(\alpha) = \frac{1 + \kappa}{\kappa} \cdot \alpha \cdot F(\alpha)
\end{gather}

almost everywhere (except points of positive measure $\mu$). 

Now let's do the same for measures $\mu_1$ and $\mu_2$ and get that equality

\begin{gather*}
F_1(\alpha) = F_2(\alpha)
\end{gather*}

holds everywhere on $[0, \alpha_{max}]$ except at most countable set of points, but $F_1$, $F_2$ are right-continuous, which implies $\mu_1 = \mu_2$

\end{proof}

\section{Stability of measure, found from given curve} 

After getting a unique solution for $V(\alpha) = G(h(\alpha) + (TV)(\alpha))$ we are interested in its dependence on the function $G$. 

\begin{thm}
Let $ \alpha_{max} $ be a positive number, and $ G_{ 1,2 } $ be two nonnegative monotonic bounded functions on $ [ 0 , \alpha_{ max} ] $, and the operator $T$ be the one from $\eqref{tubes7}$. Let $h_1, h_2$ be the functions defined by (\ref{h}) for $ G = G_1 $ and $ G = G_2 $, respectively, and $V_1$ and $V_2$ satisfy 
\begin{gather}
\label{thm_3_1}
V_1(\alpha) = G_1(h_1(\alpha) + (TV_1)(\alpha))\\
\label{thm_3_2}
V_2(\alpha) = G_2(h_2(\alpha) + (TV_2)(\alpha)) . 
\end{gather} 
Then 
\[ \|V_1 - V_2\|_{L^\infty([0, \alpha_{max}])} \le \dfrac{1 + \kappa}{2\kappa}\left(\alpha_{max} + \dfrac{3 + \kappa}{1 + \kappa}\right) \|G_1 - G_2\|_{L^\infty([0, V_{max}])} . \]
\end{thm}

\begin{proof}
First, denote $ \delta = \|G_1 - G_2\|_{L^\infty([0, V_{max}])} $ and subtract $\eqref{thm_3_2}$ from $\eqref{thm_3_1}$:
\begin{gather*}
\left|V_1(\alpha) - V_2(\alpha)\right| = \left|G_1(h_1(\alpha) + (TV_1)(\alpha)) - G_2(h_2(\alpha) + (TV_2)(\alpha))\right| \le \\ 
\le \left|G_1(h_1(\alpha) + (TV_1)(\alpha)) - G_1(h_2(\alpha) + (TV_2)(\alpha))\right| + \\ + \left|G_1(h_2(\alpha) + (TV_2)(\alpha)) - G_2(h_2(\alpha) + (TV_2)(\alpha))\right| \le \\ 
\le \left|h_1(\alpha) - h_2(\alpha)\right| + \left|T(V_1 - V_2)(\alpha)\right| + \delta .
\end{gather*}
Repeating the calculations from the proof of theorem 1 we get
\begin{gather*}
\left|T(V_1 - V_2) (\alpha)\right| \le \frac{1 - \kappa}{1 + \kappa} \cdot \|V_1 - V_2\|_{L^{\infty}}
\end{gather*}

And since $\dfrac{1 - \kappa}{1 + \kappa} < 1$:
\begin{gather*}
\|V_1 - V_2\|_{L^\infty} \cdot \left(1 - \frac{1 - \kappa}{1 + \kappa}\right) \le \|h_1 - h_2\|_{L^\infty} + \delta \\
\|V_1 - V_2\|_{L^\infty} \le \frac{1 + \kappa}{2\kappa}(\|h_1 - h_2\|_{L^\infty} + \delta)
\end{gather*}

It remains to estimate the difference $ h_1 - h_2 $,

\begin{gather*}
\left|h_1(\alpha) - h_2(\alpha)\right| \le \frac{\kappa}{1 - \kappa^2} \left(\alpha_{max} -  R ( \alpha ) \right) \left(\frac{1 + \kappa}{\kappa} - 2 \right) \left|G_1(V_{max}) - G_2(V_{max})\right| + \\ + \frac{\kappa}{1 - \kappa^2} \frac{\left( \alpha_{max} - R ( \alpha ) \right)^2}{\alpha_{max} R (\alpha)}\left|G_1(V_{max}) - G_2(V_{max})\right| = \\ 
= \frac{\kappa}{1 - \kappa^2} \left[\left(\alpha_{max} - R ( \alpha )\right) \frac{1 - \kappa}{\kappa} + \frac{\left( \alpha_{max} - R ( \alpha ) \right)^2}{\alpha_{max} R (\alpha)} \right] \left|G_1(V_{max}) - G_2(V_{max})\right| \le \\ 
\le \frac{\kappa}{1 - \kappa^2} \left(\alpha_{max} \frac{1 - \kappa }{\kappa} + \frac{2 - 2\kappa}{\kappa}\right) \cdot \delta \le \left(\alpha_{max} + \frac{2}{1 + \kappa}\right) \delta
\end{gather*}
On substituting we have
\begin{gather*}
\|V_1 - V_2\|_{L^\infty} \le \frac{1 + \kappa}{2\kappa}\left(\alpha_{max} + \frac{3 + \kappa}{1 + \kappa}\right) \delta , 
\end{gather*}
and the proof is completed.
\end{proof}

This result alone doesn't imply the stability of the measure. As will be shown we also have to estimate $\|V_1' - V_2'\|_{L^\infty(0, \alpha_{max})}$ and $\|V_1'' - V_2''\|_{L^\infty(0, \alpha_{max})}$. Let us start with a technical result. 

\begin{lm}
Let $\alpha_{min} \in ( 0 , \alpha_{ max}) $. Then formula (\ref{tubes7}) defines a bounded operator from $ L^\infty(\alpha_{min}, \alpha_{max})$ to $ C(\alpha_{min}, \alpha_{max})$, from $C(\alpha_{min}, \alpha_{max})$ to $ C^1(\alpha_{min}, \alpha_{max})$, and from $ C^1(\alpha_{min}, \alpha_{max})$ to $ C^2(\alpha_{min}, \alpha_{max} ) $.
\end{lm}

\begin{proof}
The function
\begin{gather*}
K(y, \alpha) = \kappa \cdot (1 - \kappa^2) \cdot \frac{\alpha^4}{y^2 \cdot (y^2 - (1 - \kappa^2) \cdot \alpha^2)^{\frac{3}{2}}} 
\end{gather*}
belongs to $C^\infty\left(\overline \Delta \right)$, where $ \Delta = \{ ( x , y ) \colon \alpha_{min } \le x \le y \le \alpha_{ max } \} $. Now we can rewrite our operator as follows:

\begin{gather*}
(TV)(\alpha) = \int\limits_{\alpha}^{\alpha_{max}} V(y) \cdot K(y, \alpha)\ dy
\end{gather*}
The assertions of the lemma now follow from the elementary continuity properties of integral with respect to a parameter and the consequential differentiation of this formula.
\end{proof}

\begin{lm}
Let $h_1, h_2$ be the functions defined by (\ref{h}) for $ G = G_1 $ and $ G = G_2 $, respectively. If $|G_1(V_{max}) - G_2(V_{max})| \le \delta$ then for every $k \in \mathbb{N} \cup \{0\}$ there exists a $c_k = c_k(\alpha_{max}, V_{max}, \kappa)$ such that 
\[ \left\|\dfrac{d^k h_1}{d\alpha^k} - \dfrac{d^k h_2}{d\alpha^k}\right\|_{L^\infty([0, \alpha_{max}])} \le c_k \cdot \delta \] and \[ \left\|\dfrac{d^k h_2}{d\alpha^k}\right\|_{L^\infty([0, \alpha_{max}])} \le c_k . \]
\end{lm}

\begin{proof}
It is easy to see that 
\begin{gather*}
h_1(\alpha) = p(\alpha) V_{max} + q(\alpha) G_1(V_{max})\\
h_2(\alpha) = p(\alpha) V_{max} + q(\alpha) G_2(V_{max})
\end{gather*}
for some $p, q \in C^\infty([0, \alpha_{max}])$. Hence 

\begin{gather*}
\left|\frac{d^k h_1}{d\alpha^k}(\alpha) - \frac{d^k h_2}{d \alpha^k}(\alpha)\right| \le \|q\|_{C^k([0, \alpha_{max}])} \cdot |G_1(V_{max}) - G_2(V_{max})| \le c_k \delta
\end{gather*}

And

\begin{gather*}
\left|\frac{d^k h_2}{d\alpha^k}(\alpha)\right| \le \|q\|_{C^k([0, \alpha_{max}])} \cdot |G_2(V_{max})| + \|p\|_{C^k([0, \alpha_{max}])} \cdot V_{max} \le c_k
\end{gather*}

since $G_2(V_{max}) \le V_{max}$.

\end{proof}


\begin{prop}

If $G_1, G_2 \in C([0, V_{max}])$, $\|G_1 - G_2\|_{L^\infty([0, V_{max}])} \le \delta$, $\|G_1' - G_2'\|_{L^\infty((0, V_{max}))} \le \delta$, $\lip(G_2) < \infty$ and  

\begin{gather*}
V_1(\alpha) = G_1(h_1(\alpha) + (TV_1)(\alpha))\\
V_2(\alpha) = G_2(h_2(\alpha) + (TV_2)(\alpha))
\end{gather*}

are such that $\supp(V_1), \supp(V_2) \subset [\alpha_{min}, \alpha_{max}]$, then there exists $c = c(\alpha_{max}, V_{max}, \lip(G_2'), \alpha_{min}, \kappa)$ such that $\|V_1' - V_2'\|_{L^\infty} \le c \cdot \delta$

\end{prop}

\begin{proof}

We will start proving this proposition very similar to the previous theorem, but this time we will differentiate $\eqref{thm_3_1}$ and $\eqref{thm_3_2}$ and then subtract first from second.

\begin{gather}
\left|V_1'(\alpha) - V_2'(\alpha)\right| = \left|G_1'(h_1(\alpha) + (TV_1)(\alpha)) \cdot \left(h_1'(\alpha) + \frac{d(TV_1)}{d\alpha}(\alpha)\right) - \right. \nonumber \\ 
\left. - G_2'(h_2(\alpha) + (TV_2)(\alpha)) \cdot \left(h_2'(\alpha) + \frac{d(TV_2)}{d\alpha}(\alpha)\right)\right| \le \nonumber \\ 
\left|G_1^\prime (h_1(\alpha) + (TV_1)(\alpha)) \left (h_1'(\alpha) + \frac{d(TV_1)}{d\alpha}(\alpha) - \left( h_2'(\alpha) + \frac{d(TV_2)}{d\alpha}(\alpha)\right) \right) \right| + \nonumber \\ 
\left|\left[ G_1^\prime (h_1(\alpha) + (TV_1)(\alpha)) - G_2^\prime (h_1(\alpha) + (TV_1)(\alpha)) \right] (h_2'(\alpha) + \frac{d(TV_2)}{d\alpha}(\alpha))\right| + \nonumber \\
\left| \left( G_2'(h_2(\alpha) + (TV_2)(\alpha)) - G_2'(h_1(\alpha) + (TV_1)(\alpha) )  \right) \left( h_2'(\alpha) + \frac{d(TV_2)}{d\alpha}(\alpha) \right)\right| \le \nonumber \\ 
|G_1'(h_1(\alpha) + (TV_1)(\alpha))| \left(|h_1'(\alpha) - h_2'(\alpha)| + \left|\frac{d(T( V_1- V_2))}{d\alpha}(\alpha)\right| \right) + \nonumber \\ 
|G_2'(h_1(\alpha) + (TV_1)(\alpha)) - G_1'(h_1(\alpha) + (TV_1)(\alpha))| \cdot \left|h_2'(\alpha) + \frac{d(TV_2)}{d\alpha}(\alpha)\right| + \nonumber \\
|G_2'(h_1(\alpha) + (TV_1)(\alpha)) - G_2'(h_2(\alpha) + (TV_2)(\alpha))| \cdot \left|h_2'(\alpha) + \frac{d(TV_2)}{d\alpha}(\alpha)\right| \le \nonumber \\ 
\le 1 \cdot |h_1'(\alpha) - h_2'(\alpha)| + 1 \cdot \left|\frac{d(T( V_1- V_2))}{d\alpha}(\alpha)\right| + \nonumber \\
\|G_1' - G_2'\|_{L^\infty} \cdot \left|h_2'(\alpha) + \frac{d(TV_2)}{d\alpha}(\alpha)\right| + \nonumber \\
+ \lip(G_2') \cdot \left(|h_1(\alpha) - h_2(\alpha)| + \left|(TV_1)(\alpha) - (TV_2)(\alpha)\right|\right) \cdot \left|h_2'(\alpha) + \frac{d(TV_2)}{d\alpha}(\alpha))\right| \label{differ}
\end{gather}

Let us now estimate the three terms in the right hand side separately. Since $G_1$ and $G_2$ are continuous equations $\eqref{thm_3_1}, \eqref{thm_3_2}$ imply that $V_1, V_2$ are continuous. According to lemma 4 there exists $c_1$ such that $|(TV_1)(\alpha) - (TV_2)(\alpha)| \le c_1 \cdot \delta$, $\left|\dfrac{d(TV_1)}{d\alpha}(\alpha) - \dfrac{d(TV_2)}{d\alpha}(\alpha)\right| \le c_1 \cdot \delta$ and $\left|\dfrac{d(TV_2)}{d\alpha}(\alpha)\right| \le c_1 \cdot V_{max}$.

Also, by using lemma 5 we get constant $c_2$ such that $|h_1(\alpha) - h_2(\alpha)| \le c_2 \cdot \delta$, $|h_1'(\alpha) - h_2'(\alpha)| \le c_2 \cdot \delta$ and $|h_2'(\alpha)| \le c_2$. 

And now for $c = c_1 \cdot (2 + \lip(G_2') \cdot (c_1 \cdot V_{max} + c_2 )) + c_2 \cdot (1 + V_{max} + \lip(G_2') \cdot (c_1 \cdot V_{max} + c_2))$ statement of the theorem is clear.

\end{proof}

Now we have estimation for $L^\infty$ norm of function $F_1 - F_2$, which still is not enough for estimating norm of measure $\mu_1 - \mu_1$. To answer main question this section we will state following theorem:

\begin{thm}
Assume that $G_1, G_2 \in C^2(0, V_{max})$, $\|G_1 - G_2\|_{C^2(0, V_{max})} \le \delta$, and moreover $\lip(G_2'') < \infty$. Solutions $V_1, V_2$ of $\eqref{thm_3_1}$, $\eqref{thm_3_2}$ are such that $\supp(V_1), \supp(V_2) \subset [\alpha_{min}, \alpha_{max}]$ (for some $\alpha_{min} > 0$), then there exists $c = c(\alpha_{max}, V_{max}, \lip(G_2''), \alpha_{min}, \kappa)$ such that $|\mu_1 - \mu_2| \le c \cdot \delta$.
\end{thm}

\begin{proof}
Since $G_1, G_2 \in C^2$ equations $\eqref{thm_3_1}$, $\eqref{thm_3_2}$ gives us that $V_1, V_2 \in C^2(0, \alpha_{max})$. According to $\eqref{tubes8}$ we can see that $F_1, F_2 \in C^1(0, \alpha_{max})$. Denoting $f_1(\alpha) = F_1'(\alpha) \cdot \alpha$, $f_2(\alpha) = F_2'(\alpha) \cdot \alpha$ we get $\mu_1 = f_1 \cdot \lambda$, $\mu_2 = f_2 \cdot \lambda$, where $\lambda$ is Lebesgue measure on $[0, \alpha_{max}]$. By differentiating $\eqref{tubes8}$ we get: 

\begin{gather*}
f_1(\alpha) = \frac{1 + \kappa}{\kappa} \cdot \alpha \cdot \left(\frac{V_1'(\alpha)}{\alpha}\right)'\\
f_2(\alpha) = \frac{1 + \kappa}{\kappa} \cdot \alpha \cdot \left(\frac{V_2'(\alpha)}{\alpha}\right)'
\end{gather*}

Subtracting second equation from first we get

\begin{gather*}
f_1(\alpha) - f_2(\alpha) = \frac{1 + \kappa}{\kappa} \cdot \frac{1}{\alpha} \cdot ((V_1'' - V_2'') \cdot \alpha - (V_1' - V_2'))
\end{gather*}

According to the proposition we already know that $\|V_1' - V_2'\|_{L^\infty} \le \tilde c \cdot \delta$ and if we get similar estimation for $\|V_1'' - V_2''\|_{L^\infty}$ we will get constant $c$ such that $\|f_1 - f_2\|_{L^\infty} \le c \cdot \delta$ and hence $|\mu_1 - \mu_2| \le c \cdot \alpha_{max} \cdot \delta$, which will finish the proof. 

To get this estimation we will use the same scheme as in the proof of proposition: we will take second derivative of $\eqref{thm_3_1}$, $\eqref{thm_3_2}$ and subtract second from first:

\begin{gather*}
|V_1''(\alpha) - V_2''(\alpha)| \le \left|G_1''(h_1(\alpha) + (TV_1)(\alpha)) \cdot \left(h_1'(\alpha) + \frac{d(TV_1)}{d\alpha}(\alpha)\right)^2 - \right. \\ 
\left. - G_2''(h_2(\alpha) + (TV_2)(\alpha)) \cdot \left(h_2'(\alpha) + \frac{d(TV_2)}{d\alpha}(\alpha)\right)^2 \right| + \\
+ \left|G_1'(h_1(\alpha) + (TV_1)(\alpha)) \cdot \left(h_1''(\alpha) + \frac{d^2(TV_1)}{d\alpha^2}(\alpha)\right) - \right. \\ 
\left. - G_2'(h_2(\alpha) + (TV_2)(\alpha)) \cdot \left(h_2''(\alpha) + \frac{d^2(TV_2)}{d\alpha^2}(\alpha)\right)\right|
\end{gather*}

The first term can be estimated for the same reasoning as in proposition, but this time constant will depend on $\lip(G_2'')$. For the second term we will use the same method as in preposition, and it is easy to see that in addition to estimations shown in preposition we also need is to get estimations $\|h_1'' - h_2''\|_{L^\infty} \le c \cdot \delta$, $\left\|\dfrac{d^2(TV_1)}{d\alpha^2}(\alpha) - \dfrac{d^2(TV_2)}{d\alpha^2}(\alpha)\right\|_{L^\infty} \le c \cdot \delta$, $\|h_2''\|_{L^\infty} \le c$ and $\left\|\dfrac{d^2(TV_2)}{d\alpha^2}(\alpha)\right\|_{L^\infty} \le c$ for some $c = c(\alpha_{max}, V_{max}, \lip(G_2''), \alpha_{min}, \kappa)$, which directly follows from lemma 4 and lemma 5.


\end{proof}

Note that all constants mentioned above can be calculated explicitly.

\section{Numerical simulation}

In previous section we have showed that procedure of reestablishing the measure given a curve is stable in the mathematical sense, but in the context of petroleum engineering we are interested in the explicit constant in estimations given above. In other words, we would like to estimate the ratio of difference of measures to difference of derivatives of functions $G_1$ and $G_2$. To be more precise we will calculate following:

\begin{gather*}
c = \left\|\frac{G_1' - G_2'}{G_1' + G_2'} \right \|_{L^1(0, V_{max})}: \left\| \frac{F_1 - F_2}{F_1 + F_2}\right\|_{L^1(0, \alpha_{max})}
\end{gather*}
,where $F_j (\alpha) = \mu_j(0, \alpha)$.

To do this we have performed a numerical simulation an calculated constant described above for some pairs of measures $\mu_1$, $\mu_2$. Here is a brief description of algorithm we have used.

We generate these measures as a number of random delta-measures on section $\left[2.5, 10\right)$ taking $\alpha_{max} = 10$. Then we calculate $V_{j, max} = V_{j, o}(\alpha_{max}) + V_{j, w}(\alpha_{max})$, and if $|V_{1, max} - V_{2, max}| < \dfrac{V_{1, max}}{10}$ we simply calculate $G_1$ and $G_2$ according to formulae given in section 4. In following calculations it is important that $G_1'$ and $G_2'$ can be found given only $V_{w, 1}'$, $V_{o, 1}'$ and $V_{w, 2}'$, $V_{o, 2}'$ respectively. Also it is easy to see that these derivatives can be calculated as integrals $\int g(\alpha, y) d\mu_j(y)$ for appropriate functions $g$ and limits of integration. 

On figure \ref{picture_1} results of numerical simulation are represented. For measures shown constant $c$ reaches value of 5 and higher.

\begin{figure}[ht]
\centering
\includegraphics[width=15 cm]{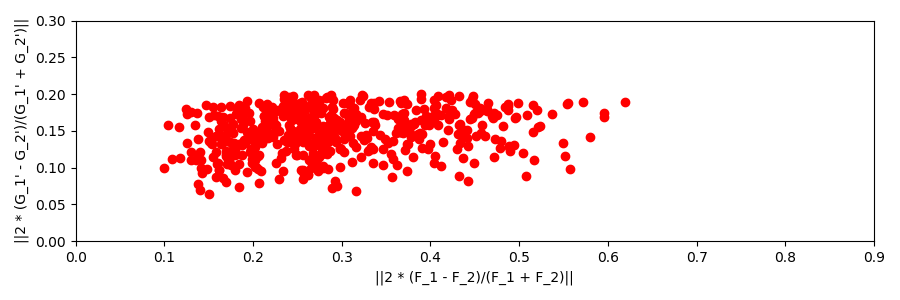}

\caption{Numerical simulation results}
\label{picture_1}
\end{figure}

\end{document}